\documentclass[12pt]{amsart}
\usepackage{mathrsfs}
\usepackage{amssymb}
\usepackage{amsbsy}
\newtheorem{theorem}{Theorem}[section]
\newtheorem{lemma}[theorem]{Lemma}
\newtheorem{corollary}[theorem]{Corollary}
\newtheorem{proposition}[theorem]{Proposition}

\theoremstyle{definition}

\newtheorem{example}[theorem]{Example}

\theoremstyle{remark}

\numberwithin{equation}{section}

\usepackage{cite}

\begin{document}

\begin{center}
{\Large \bf  A unified treatment for  $\boldsymbol{L_p}$
Brunn-Minkowski\\ type inequalities}
\end{center}

\vskip 20pt

\begin{center}
{{\bf Du~~Zou$^1$,~~~~~~~~~Ge~~Xiong$^2$}\\~~ \\1. Department of Mathematics,
Wuhan University of Science and Technology, \\ Wuhan, 430081, PR China\\
2. Department of Mathematics, Tongji University, Shanghai, 200092, PR China}
\end{center}

\vskip 10pt

\footnotetext{E-mail addresses: zoudu@wust.edu.cn; xiongge@tongji.edu.cn} \footnotetext{Research of the authors was supported by NSFC No. 11471206.}

\begin{center}
\begin{minipage}{14cm}
{\bf Abstract}  A unified approach used to generalize classical
Brunn-Minkowski type inequalities to $L_p$ Brunn-Minkowski type
inequalities, called the $L_p$ transference principle, is refined in
this paper. As illustrations of the effectiveness and practicability
of this method, several new  $L_p$ Brunn-Minkowski type inequalities
concerning the mixed volume, moment of inertia, quermassintegral,
projection body and capacity are established.



\vskip 10pt
\end{minipage}
\end{center}

\vskip 25pt
\section{\bf Introduction}
\vskip 10pt

The classical Brunn-Minkowski inequality is a marvelous result of
combining two basic notions: vector addition and volume, which
reads as follows: If $K$ and $L$ are convex bodies (compact convex
sets with nonempty interiors) in Euclidean $n$-space $\mathbb{R}^n$
and $\alpha\in (0,1)$, then
\begin{equation}\label{classicalBMineq}
V_n\left((1-\alpha)K +\alpha L \right)^{\frac{1}{n}} \ge (1-\alpha)
V_n\left( K \right)^{\frac{1}{n}} + \alpha V_n\left( L \right)^{\frac{1}{n}},
\end{equation}
where
\[ (1-\alpha)K +\alpha L =\{(1-\alpha)x +\alpha y: x\in K, y\in L\}, \]
and $V_n$ denotes the $n$-dimensional volume. Equality holds in
(\ref{classicalBMineq}) if and only if $K$ and $L$ are homothetic
(i.e., they coincide up to a translation and a dilate). In brief,
the functional ${V_n}^{1/n}$ from $\mathcal{K}^n$, the class of
convex bodies in $\mathbb{R}^n$, to $[0,\infty)$ is concave.

As one of the cornerstones of  convex geometry (see Gardner
\cite{BookGardner}, Gruber \cite{BookGrubder}, and Schneider
\cite{BookSchneider}), the Brunn-Minkowski inequality is a powerful
tool for solving many extremum problems dealing with important
geometric quantities such as volume and surface area. It is also
related closely with many other fundamental inequalities, such as
the classical isoperimetric inequality, the
Pr$\rm{\acute{e}}$kopa-Leindler inequality, the Sobolev inequality,
and the Brascamp-Lieb inequality. See, e.g., Barthe \cite{Barthe}
and Bobkov and Ledoux \cite{BMineqBobkovLedoux}. For a more
comprehensive understanding, we refer to the excellent survey of
Gardner \cite{Gardner1}.

Nearly half a century ago, Firey \cite{Firey} (see also Gardner,
Hug, and Weil \cite[p. 2311]{Gardner4}, Lutwak, Yang, and Zhang
\cite{LpBMLYZ5}, and Schneider \cite[Section 9.1]{BookSchneider})
introduced the $L_p$ addition of convex bodies and established the
$L_p$ Brunn-Minkowski inequality for this new operation: If $K$ and
$L$ are convex bodies in $\mathbb{R}^n$ containing the origin in
their interiors, $p\in (1,\infty)$ and $\alpha\in (0,1)$, then
\begin{equation}\label{LpBMineq}
V_n\left( (1-\alpha)\cdot_p K +_p \alpha\cdot_p L
\right)^{\frac{p}{n}} \ge (1-\alpha) V_n\left( K
\right)^{\frac{p}{n}} + \alpha V_n\left( L\right)^{\frac{p}{n}},
\end{equation}
where $(1-\alpha)\cdot_p K=(1-\alpha)^{1/p}K$, $\alpha\cdot_p
L=\alpha^{1/p} L$, and
\[(1-\alpha)\cdot_p K +_p \alpha\cdot_p L= \left\{ {{{(1 - \beta )}^{\frac{{p - 1}}{p}}}{(1-\alpha)}^{\frac{1}{p}}x + {\beta ^{\frac{{p - 1}}{p}}}\alpha^{\frac{1}{p}}y:x \in K,y \in L,0 \le \beta  \le 1} \right\}.\]
Equality holds in (\ref{LpBMineq}) if and only if $K$ and $L$ are
dilates. Write $\mathcal{K}^n_o$  for the class of convex bodies
with the origin in their interiors. In brief, the functional
${V_{n}}^{{p}/{n}}$ from  $\mathcal{K}^n_o$  to $[0,\infty)$ is concave.

Further developments of the $L_p$  Brunn-Minkowski theory were greatly
impelled by Lutwak \cite{LpBMLutwak1, LpBMLutwak2}, who nearly set
up a broad framework for the theory. A series
of fundamental notions, geometric objects, and central results in
the classical Brunn-Minkowski theory  evolved into their $L_p$
analogs. See, e.g., \cite{LpBMLYZ1, LpBMLYZ2, LpBMLYZ3, LpBMLYZ4,
LpBMLYZ5, LpBMPaourisWerner, LpBMRyaboginZvavitch, LpBMSchuttWerner,
LpBMWernerYe, LpBMYe}.

In retrospect, we observe that to establish new Brunn-Minkowski type
inequalities, we  encounter essentially the following  two general
situations.

First,  if a  functional $F: \mathcal{K}^n\to [0,\infty)$ is
positively homogeneous of order $j$, $j\neq 0$, is it the case that
functional $F^{{1}/{j}}$ is concave?  Precisely, for
$K,L\in\mathcal{K}^n$ and $\alpha\in (0,1)$, is it the case that
\[ F \left( (1-\alpha)K+\alpha L \right)^{\frac{1}{j}}\ge (1-\alpha) F
\left(K \right)^{\frac{1}{j}} +\alpha F \left(L \right)^{\frac{1}{j}}? \]
Incidentally, we can list some beautiful
confirmed examples within the classical Brunn-Minkowski theory, such
as the classical mixed volumes (see Schneider \cite[p.
406]{BookSchneider}), Hadwiger's harmonic quermassintegrals (see,
e.g., Hadwiger \cite[p. 268]{BMineqHadwiger} and Schneider \cite[p.
514]{BookSchneider}), and Lutwak's affine quermassintegrals (see,
e.g., Gardner \cite[p. 393]{Gardner1},
\cite[p. 361]{GadnerAffineQuermass} and  Schneider \cite[p.
515]{BookSchneider}). Also, some instances were discovered in other
disciplines. For example, Brascamp and Lieb
\cite{BMineqBrascampLieb} established a Brunn-Minkowski type
inequality for the first eigenvalue of the Laplace operator. Borell
\cite{BMineqBorell} proved a Brunn-Minkowski type inequality for the
Newtonian capacity. See also Caffarelli, Jerison, and Lieb
\cite{BMineqCaffarelliJerison}, Colesanti and Salani
\cite{BMineqColesanti1}, Gardner and Hartenstine \cite{Gardner3},
and the references within.

Second, assume that a functional $F: \mathcal{K}^n\to [0,\infty)$ is
positively homogeneous and concave, for $K,L\in\mathcal{K}^n_o$ and
$\alpha\in (0,1)$. Is it the case that
\[F\left( (1-\alpha)\cdot_p K +_p \alpha\cdot_p L \right)^{p} \ge (1-\alpha) F\left( K \right)^{p} + \alpha F\left( L\right)^{p}? \]
Obviously, if $F={V_{n}}^{1/n}$,  for $p>1$, the $L_p$ Brunn-Minkowski
inequality  is a confirmed case.

The prime motivation of this paper is to formulate and prove the
following \emph{$L_p$ transference principle}.

\begin{theorem}
Suppose that $F:\mathcal{K}^n\to [0,\infty)$ is positively
homogeneous,  increasing and concave, and $p\in (1,\infty)$.  If $K,
L\in\mathcal{K}^n_o$, then
\[ F((1-\alpha)\cdot_p K+_p \alpha\cdot_p L)^p\ge (1-\alpha)F(K)^p+\alpha F(L)^p, \quad  for\; all\;  \alpha \in (0,1). \]
Furthermore, if $F:\mathcal{K}^n_o\to (0,\infty)$ is strictly
increasing,  equality holds if and only if $K$ and $L$ are dilates.
\end{theorem}

See Section 2 for the definitions of positive homogeneity and
increasing property  of a functional $F$.

In Section 3, we prove Theorem 1.1 and then  dwell on the equality
condition. It is observed that  equality holds in the classical
Brunn-Minkowski inequality if and only if the  convex bodies are
\emph{homothetic}, while   equality holds in the $L_p$
Brunn-Minkowski inequality  if and only if the  convex bodies are
\emph{dilates}. We reveal the reason and characterize this
phenomenon.  See Theorem 3.4 and Theorem 3.5.

In Section 4, as illustrations of the effectiveness and
practicability of our $L_p$ transference principle,  several new
$L_p$ Brunn-Minkowski type inequalities are established, which are
concerned with the classical mixed volume, moment of inertia, affine
quermassintegral, harmonic quermassintegral,  projection body and
capacity. For example, by using the $L_p$ transference principle, we
obtain the  $L_p$ capacitary Brunn-Minkowski inequality directly.

\begin{theorem}
Suppose that $K,L\in\mathcal{K}^n_o$, $1\le q<n$, and $1<p<\infty$.
Then
\[{\rm{Ca}}{{\rm{p}}_q}{(K{ + _p}L)^{\frac{p}{{n - q}}}} \ge {\rm{Ca}}{{\rm{p}}_q}{(K)^{\frac{p}{{n - q}}}} + {\rm{Ca}}{{\rm{p}}_q}{(L)^{\frac{p}{{n - q}}}},\]
with equality if and only if $K$ and $L$ are dilates.
\end{theorem}

\vskip 25pt
\section{\bf Preliminaries}
\vskip 10pt

As usual, $S^{n-1}$ denotes the unit sphere in $n$-dimensional
Euclidean space $\mathbb{R}^n$, and $B^n$ denotes the unit ball in
$\mathbb{R}^n$. If $x,y\in\mathbb{R}^n$, then $x\cdot y$ denotes the
inner product of $x$ and $y$. If $u\in S^{n-1}$, then $u^\bot$
denotes the $(n-1)$-dimensional subspace orthogonal to $u$. Write
$V_j$ for $j$-dimensional volume, where $j=1,\ldots,n$. As usual,
$\omega_j$ denotes the volume of $j$-dimensional unit Euclidean ball.

Write ${\rm G}_{n,j}$ for the Grassmann manifold of all
$j$-dimensional linear subspaces of $\mathbb{R}^n$, which is
equipped with Haar probability measure $\mu_{j}$. For
$K\in\mathcal{K}^n$, let $K|\xi$ be the orthogonal projection of $K$
onto $\xi\in {\rm G}_{n,j}$.

Each convex body $K$ in $\mathbb{R}^n$ is uniquely determined by its
support function $h_K: \mathbb{R}^n\to\mathbb{R}$, which is defined
by
\[ h_K(x)=\max\left\{x\cdot y: y\in K \right\},\]
for $x\in\mathbb{R}^n$. For $\alpha>0$,  the body $\alpha K=\{\alpha
x: x\in K\}$ is called a dilate of $K$.

For $K, L\in \mathcal{K}^n$, their Minkowski sum is the convex body
\[K+L=\{x+y: x\in K, y\in L\}.\]

Let $1<p < \infty$. The $L_p$ sum of $K, L\in \mathcal{K}^n_o$  is
the convex body $K +_p L$, defined by
\[ h_{K +_p L}(u)^p =h_{K}(u)^p +h_{L}(u)^p, \]
for $u\in S^{n-1}$. If  $p=\infty$, the convex body $K +_\infty L$
is defined by
\[h_{K +_\infty L}(u)=\max\{ h_{K}(u), h_{L}(u)\}, \]
for $u\in S^{n-1}$.

For $\alpha>0$ and $K\in\mathcal{K}^n_o$, the $L_p$ scalar
multiplication $\alpha\cdot_p K$ is the convex body
$\alpha^{\frac{1}{p}} K$.

Given a functional $F:\mathcal{K}^n\to [0,\infty)$, we say that $F$
is

\noindent (1) \emph{positively homogeneous}, provided
\[F(\alpha K)=\alpha F(K),\]
for $K\in\mathcal{K}^n$ and $\alpha>0$.

\noindent (2) \emph{increasing}, provided
\[ F(K)\le F(L),\]
for $K,L\in\mathcal{K}^n$ with $K\subseteq L$. Moreover, if the
strict inclusion $K\subsetneq L$ implies  $F(K)< F(L)$, then $F$ is
\emph{strictly increasing}.

\noindent (3) \emph{$p$-concave}, provided
\[ F((1-\alpha)\cdot_p K+_p\alpha\cdot_p L)\ge ((1-\alpha)F(K)^p+\alpha F(L)^p)^{\frac{1}{p}},\]
for $K,L\in\mathcal{K}^n$ and $\alpha\in (0,1)$. As usual,
$1$-concave is called concave for brevity.

Associated with a functional $F:\mathcal{K}^n_o\to [0,\infty)$,
$p\in[1,\infty)$ and $K,L\in\mathcal{K}^n_o$, it is convenient to
introduce a function $F_{p; K, L}: [0,1]\to [0,\infty)$ defined by
\[F_{p;K,L}(\alpha)=F\left( (1-\alpha)\cdot_p K +_p \alpha\cdot_p L\right )^p, \]
for $\alpha \in [0,1]$.

The next lemma shows  that the $p$-concavity of $F$ and the
concavity of $F_{p;K,L}$ are actually equivalent.

\begin{lemma}\label{lemma2.1}
Suppose that $K,L\in\mathcal{K}^n_o$ and $1\le p<\infty$. Then, $F$
is $p$-concave, if and only if $F_{p;K,L}$ is concave.
\end{lemma}

\begin{proof}
Let $\lambda,\alpha,\beta\in [0,1]$. Assume that $F_{p;K,L}$ is
concave. Then
\[{F_{p;K,L}}((1 - \lambda )\alpha  + \lambda \beta ) \ge (1 - \lambda ){F_{p;K,L}}(\alpha ) + \lambda {F_{p;K,L}}(\beta ).\]
Taking $\alpha=0$ and $\beta=1$, we obtain
\[{F_{p;K,L}}(\lambda ) \ge (1 - \lambda ){F_{p;K,L}}(0) + \lambda {F_{p;K,L}}(1),\]
i.e.,
\[F{((1 - \lambda ){ \cdot _p}K{ + _p}\lambda { \cdot _p}L)^p} \ge (1 - \lambda )F{(K)^p} + \lambda F{(L)^p},\]
which shows that $F$ is $p$-concave.

Conversely, assume that $F$ is $p$-concave. Let
\begin{align*}
&{K_\alpha } = (1 - \alpha ){ \cdot _p}K{ + _p}\alpha { \cdot _p}L, \\
&{K_\beta } = (1 - \beta ){ \cdot _p}K{ + _p}\beta { \cdot _p}L, \\
&Q = \left( {1 - ((1 - \lambda )\alpha  + \lambda \beta )} \right){
\cdot _p}K{ + _p}((1 - \lambda )\alpha  + \lambda \beta ){ \cdot_p}L.
\end{align*}
Then
\begin{align*}
{h_Q}^p &= \left( {1 - ((1 - \lambda )\alpha  + \lambda \beta )} \right){h_K}^p + ((1 - \lambda )\alpha  + \lambda \beta ){h_L}^p \\
&= \left( {(1 - \lambda )(1 - \alpha ) + \lambda (1 - \beta )} \right){h_K}^p + ((1 - \lambda )\alpha  + \lambda \beta ){h_L}^p \\
&= (1 - \lambda )\left( {(1 - \alpha ){h_K}^p + \alpha {h_L}^p} \right) + \lambda \left( {(1 - \beta ){h_K}^p + \beta {h_L}^p} \right) \\
&= (1 - \lambda ){h_{{K_\alpha }}}^p + \lambda {h_{{K_\beta }}}^p,
\end{align*}
which implies that $Q=(1 - \lambda )\cdot_p K_\alpha +_p
\lambda\cdot_p K_\beta$.

Thus, from the  $p$-concavity of $F$, it follows that
\begin{align*}
{F_{p;K,L}}\left( {(1 - \lambda )\alpha  + \lambda \beta }\right) &= F{(Q)^p} \\
&= F{((1 - \lambda ){ \cdot _p}{K_\alpha }{ + _p}\lambda { \cdot _p}{K_\beta })^p} \\
&\ge (1 - \lambda )F{({K_\alpha })^p} + \lambda F{({K_\beta })^p} \\
&= (1 - \lambda ){F_{p;K,L}}(\alpha) + \lambda {F_{p;K,L}}(\beta ),
\end{align*}
which shows that $F_{p;K,L}$ is concave.
\end{proof}

The following lemma will be used in Section 3.
\begin{lemma}\label{lemma2.2}
Suppose $K,L\in\mathcal{K}^n_o$, $1<p<\infty$, and $0< \alpha< 1$.
Then
\[ (1-\alpha)\cdot_p K+_p \alpha\cdot_p L\supseteq (1-\alpha) K + \alpha L, \]
with equality if and only if $K=L$.
\end{lemma}

\begin{proof}
From the definition of $(1-\alpha)\cdot_p K+_p \alpha\cdot_p L$ and
the strict convexity of  $f(t)=t^p$ in $t\in (0,\infty)$, it follows
that for $u\in S^{n-1}$,
\begin{align*}
{h_{(1 - \alpha){ \cdot _p}K{ + _p}\alpha{ \cdot _p}L}}(u) &= {\left( {(1 - \alpha){h_K}{{(u)}^p} + \alpha{h_L}{{(u)}^p}} \right)^{\frac{1}{p}}} \\
&\ge (1 - \alpha){h_K}(u) + \alpha{h_L}(u) \\
&= {h_{(1 - \alpha)K + \alpha L}}(u).
\end{align*}
Equality holds in the second line if and only if
$h_K(u)=h_L(u)$, for all $u\in S^{n-1}$, and therefore if and only
if $K=L$.
\end{proof}

The next lemma will be needed in Section 3 and  Section 4.
\begin{lemma}\label{lemma2.3}
Suppose $K,L\in\mathcal{K}^n$ and $j\in\{1,\ldots,n-1\}$. If
$K\subsetneq L$, then there exists a $\mu_j$-measurable subset
$G\subseteq {\rm G}_{n,j}$ such that $\mu_j(G)>0$ and
\[ V_j(K|\xi)< V_j(L|\xi),\quad for\;all\; \xi\in G. \]
\end{lemma}

\begin{proof}
Recall that $h_K$ and $h_L$ are continuous on $S^{n-1}$. So, the
assumption $K\subsetneq L$ implies that there exists an open
geodesic ball $U$ in $S^{n-1}$ such that $\mathcal{H}^{n-1}(U)>0$
and $ h_K(u)<h_L(u)$, for all $u\in U$. Let
\[ G=\{\xi\in G_{n,j}: \xi\cap U\neq\emptyset\}. \]

Next, we aim  to show that
\[ \mu_j(G)>0. \]
Note that when its points are antipodally identified, the sphere
$S^{n-1}$ is identified with an $(n-1)$-dimensional elliptic space
of constant curvature one. The measure $\mu_j(G)$ can be represented
as
\[{\mu _j}(G) = \frac{{(n - j)!j!{\omega _j} \cdots {\omega _1}}}{{n!{\omega _n} \cdots {\omega _{n - j + 1}}}}\int_{G \cap {L_{j - 1}}\neq\emptyset} {d{L_{j - 1}}} ,\]
where $dL_{j-1}$ denotes the kinematic density of a moving
$(j-1)$-dimensional plane $L_{j-1}$ in the elliptic space $S^{n-1}$.
For more details, see Santal$\rm \acute{o}$ \cite[pp.
299-310]{BookSantalo}. Let $r$ be the geodesic radius of $U$. Then
equation (17.52) in \cite{BookSantalo} shows that
\[\int_{G \cap {L_{j - 1}} \ne \emptyset } {d{L_{j - 1}}}  = \frac{{(n - 1)!{\omega _{n - 1}} \cdots {\omega _{n - j}}}}{{(j - 1)!(n - j - 1)!{\omega _{j - 1}} \cdots {\omega _1}}}\int\limits_0^r {{{(\cos t)}^{j - 1}}{{(\sin t)}^{n - j - 1}}dt} .\]
Thus, $\mu_j(G)>0$.

Finally, let $\xi\in G$ and $u\in \xi\cap U$. Since $\xi\cap
U\neq\emptyset$, the definition of $U$ implies that
\[ h_{K|\xi}(u)=h_K(u)<h_L(u)=h_{L|\xi}(u). \]
Since $h_{K|\xi}$ and $h_{L|\xi}$ are continuous on
$S^{n-1}\cap\xi$, and $h_{K|\xi}\le h_{L|\xi}$, we have
$K|\xi\subsetneq L|\xi$. This, combined with the convexity of
$K|\xi$ and $ L|\xi$, implies $V_j(K|\xi)<V_j(L|\xi)$.
\end{proof}

\vskip 25pt
\section{\bf The $\boldsymbol{L_p}$ transference principle}
\vskip 10pt

\subsection{Statement}

In the following, we prove the \emph{$L_p$ transference principle}.
\begin{theorem}\label{theorem3.1}
Suppose that  $F:\mathcal{K}^n\to [0,\infty)$ is positively
homogeneous, increasing and concave, and $p\in (1,\infty)$. If
$K,L\in\mathcal{K}^n_o$,  then
\begin{equation}\label{ineq3.1}
F((1-\alpha)\cdot_p K+_p \alpha\cdot_p L)^p\ge
(1-\alpha)F(K)^p+\alpha F(L)^p, \quad for\;all\; \alpha\in (0,1).
\end{equation}
Furthermore, if $F:\mathcal{K}^n_o\to [0,\infty)$ is strictly
increasing,  equality holds in (\ref{ineq3.1}) if and only if $K$
and $L$ are dilates.
\end{theorem}

\begin{proof}
If $F(K)F(L)=0$, then inequality (\ref{ineq3.1}) holds. To see this,
 assume $F(K)=0$. Then the definitions
of $(1-\alpha)\cdot_p K+_p\alpha\cdot_p L$ and $\alpha\cdot_p L$
directly imply that \[(1-\alpha)\cdot_p K+_p\alpha\cdot_p L\supseteq
\alpha\cdot_p L =\alpha^{\frac{1}{p}}L.\]
From the monotonicity and positive homogeneity of $F$, we have
\[F((1-\alpha)\cdot_p K+_p\alpha\cdot_p L)^p\ge F\left(\alpha^{\frac{1}{p}}L\right)^p=\alpha F(L)^p.\]

So, we  assume that $F(K)F(L)>0$. Then inequality (\ref{ineq3.1}) is
equivalent to
\[\frac{{F\left( {(1 - \alpha){ \cdot _p}K{ + _p}\alpha{ \cdot _p}L} \right)}}{{{{\left[ {(1 - \alpha)F{{(K)}^p} + \alpha F{{(L)}^p}} \right]}^{\frac{1}{p}}}}} \ge 1.\]
By the positive homogeneity of $F,$   this is equivalent to
\[F\left( {\frac{{(1 - \alpha){ \cdot _p}K{ + _p}\alpha{ \cdot _p}L}}{{{{\left[ {(1 - \alpha)F{{(K)}^p} + \alpha F{{(L)}^p}} \right]}^{\frac{1}{p}}}}}} \right) \ge 1.\]
From the definition of $L_p$ scalar multiplication,  this is
equivalent to
\begin{equation}\label{ineq3.2}
F\left( {\left( {\frac{{1 - \alpha}}{{(1 - \alpha)F{{(K)}^p} +
\alpha F{{(L)}^p}}}} \right){ \cdot _p}K{ + _p}\left(
{\frac{\alpha}{{(1 - \alpha)F{{(K)}^p} + \alpha F{{(L)}^p}}}} \right){ \cdot _p}L} \right) \ge 1.
\end{equation}

Now,  again using the definition of $L_p$ scalar multiplication, we have
\begin{align*}
&\left( {\frac{{1 - \alpha}}{{(1 - \alpha )F{{(K)}^p} + \alpha F{{(L)}^p}}}} \right){ \cdot _p}K{ + _p}\left( {\frac{\alpha}{{(1 - \alpha)F{{(K)}^p} + \alpha F{{(L)}^p}}}} \right){ \cdot _p}L \\
&\quad = \left( {\frac{{(1 - \alpha)F{{(K)}^p}}}{{(1 - \alpha)F{{(K)}^p} + \alpha F{{(L)}^p}}}} \right){ \cdot _p}\left( {\frac{K}{{F(K)}}} \right){ + _p}\left( {\frac{{\alpha F{{(L)}^p}}}{{(1 - \alpha )F{{(K)}^p} + \alpha F{{(L)}^p}}}} \right){ \cdot _p}\left( {\frac{L}{{F(L)}}} \right) \\
&\quad= (1 - \alpha'){ \cdot _p}\left( {\frac{K}{{F(K)}}} \right){ +
_p}\alpha'{ \cdot _p}\left( {\frac{L}{{F(L)}}} \right),
\end{align*}
where
\[\alpha' = \frac{{\alpha F{{(L)}^p}}}{{(1 - \alpha )F{{(K)}^p} + \alpha F{{(L)}^p}}}.\]
By Lemma 2.2, we have
\[(1 - \alpha'){ \cdot _p}\left( {\frac{K}{{F(K)}}} \right){ + _p}\alpha'{ \cdot _p}\left( {\frac{L}{{F(L)}}} \right) \supseteq (1 - \alpha')\left( {\frac{K}{{F(K)}}} \right){ +}\alpha'\left( {\frac{L}{{F(L)}}} \right).\]

Hence, from the definition of $\alpha'$, the above inclusion
together with the monotonicity of $F$, the concavity of $F$, and the
positive homogeneity of $F$, it follows that
\begin{align*}
&F\left( {\left( {\frac{{1 - \alpha}}{{(1 - \alpha)F{{(K)}^p} + \alpha F{{(L)}^p}}}} \right){ \cdot _p}K{ + _p}\left( {\frac{\alpha}{{(1 - \alpha)F{{(K)}^p} + \alpha F{{(L)}^p}}}} \right){ \cdot _p}L} \right) \\
&\quad= F\left( {(1 - \alpha'){ \cdot _p}\left( {\frac{K}{{F(K)}}} \right){ + _p}\alpha'{ \cdot _p}\left( {\frac{L}{{F(L)}}} \right)} \right) \\
&\quad\ge F\left( {(1 - \alpha')\left( {\frac{K}{{F(K)}}} \right) + \alpha'\left( {\frac{L}{{F(L)}}} \right)} \right) \\
&\quad\ge (1 - \alpha')F\left(    {\frac{K}{{F(K)}}}   \right) + \alpha'F\left(   {\frac{L}{{F(L)}}}   \right) \\
&\quad= (1 - \alpha') + \alpha' \\
&\quad= 1.
\end{align*}
This establishes inequality (\ref{ineq3.2}).  Therefore,  inequality
(\ref{ineq3.1}) holds.

Finally, under the additional assumption  that $F$ is strictly
increasing on $\mathcal{K}^n_o$, we aim to prove the equality
condition. Note that the strict monotonicity and
positive homogeneity of $F$ imply that $F$ is positive on
$\mathcal{K}^n_o$.

Assume that equality holds in (\ref{ineq3.1}). Then
\begin{align*}
&F\left( {(1 - \alpha'){ \cdot _p}\left( {\frac{K}{{F(K)}}} \right){
+ _p}\alpha'{ \cdot _p}\left( {\frac{L}{{F(L)}}} \right)} \right)\\
&\quad = F\left( {(1 - \alpha')\left( {\frac{K}{{F(K)}}} \right) +
\alpha'\left( {\frac{L}{{F(L)}}} \right)} \right).
\end{align*}
By the strict  monotonicity of $F$, this implies that
\[(1 - \alpha'){ \cdot _p}\left( {\frac{K}{{F(K)}}} \right){ + _p}\alpha'{ \cdot _p}\left( {\frac{L}{{F(L)}}} \right) = (1 - \alpha')\left( {\frac{K}{{F(K)}}} \right) + \alpha'\left( {\frac{L}{{F(L)}}} \right).\]
This equation, combined with Lemma 2.2, implies that
\[\frac{K}{{F(K)}} = \frac{L}{{F(L)}},\]
which shows that $K$ and $L$ are dilates.

Conversely, assume that $K$ and $L$ are dilates, say $K=\beta L$,
for some constant $\beta>0$. From the definition of $L_p$
combination of convex bodies,
\begin{align*}
(1 - \alpha){ \cdot _p}K{ + _p}\alpha{ \cdot _p}L &= (1 - \alpha){ \cdot _p}(\beta L){ + _p}\alpha{ \cdot _p}L \\
&= {(1 - \alpha)^{\frac{1}{p}}}\beta L{ + _p}{\alpha^{\frac{1}{p}}}L \\
&= {\left( {(1 - \alpha){\beta^p} + \alpha} \right)^{\frac{1}{p}}}L.
\end{align*}
From this and the positive homogeneity of $F$, it follows that
\begin{align*}
F{((1 - \alpha){ \cdot _p}K{ + _p}\alpha{ \cdot _p}L)^p} &= F{\left( {{{\left( {(1 - \alpha){\beta^p} + \alpha} \right)}^{\frac{1}{p}}}L} \right)^p} \\
&= (1 - \alpha){\beta^p}F{(L)^p} + \alpha F{(L)^p} \\
&= (1 - \alpha)F{(\beta L)^p} + \alpha F{(L)^p} \\
&= (1 - \alpha)F{(K)^p} + \alpha F{(L)^p},
\end{align*}
which shows that equality holds in (\ref{ineq3.1}).
\end{proof}

Theorem \ref{theorem3.1} immediately yields the following corollary.
\begin{corollary}\label{cor3.1}
Suppose that  $F:\mathcal{K}^n\to [0,\infty)$ is positively
homogeneous,  increasing and concave, and $p\in (1,\infty)$. If
$K,L\in\mathcal{K}^n_o$, then
\begin{equation}\label{ineq3.3}
F(K+_pL)^p\ge F(K)^p + F(L)^p.
\end{equation}
Furthermore, if $F:\mathcal{K}^n_o\to [0,\infty)$ is strictly
increasing,  equality holds in (\ref{ineq3.3}) if and only if $K$
and $L$ are dilates.
\end{corollary}

\begin{proof}
Let $\alpha\in (0,1)$. From the definition of $L_p$ scalar
multiplication, Theorem 3.1 and the positive homogeneity of $F$, it
follows that
\begin{align*}
F{(K{ + _p}L)^p} &= F{\left( {(1 - \alpha){ \cdot _p}({{(1 - \alpha)}^{ - \frac{1}{p}}}K){ + _p}\alpha{ \cdot _p}({\alpha^{ - \frac{1}{p}}}L)} \right)^p} \\
&\ge (1 - \alpha)F{\left( {{{(1 - \alpha)}^{ - \frac{1}{p}}}K} \right)^p} + \alpha F{\left( {{\alpha^{ - \frac{1}{p}}}L} \right)^p} \\
&= F{(K)^p} + F{(L)^p},
\end{align*}
which is precisely  (\ref{ineq3.3}).  If the monotonicity of $F$ is
strict, equality holds in the second line if and only if
$(1-\alpha)^{-{1}/{p}}K$ and $\alpha^{-{1}/{p}}L$ are dilates, and
therefore if and only if $K$ and $L$ are dilates.
\end{proof}

For convex bodies $K,L\in\mathcal{K}^n$, their Hausdorff distance is
\[ \delta_H(K,L)=\max\{|h_K(u)-h_L(u)|: u\in S^{n-1}\}. \]
From the definition of $L_p$ addition, we have $K+_p L\to K+_\infty
L$, as $p\to\infty$. Assume functional $F$ in Theorem
\ref{theorem3.1} (or Corollary \ref{cor3.1}) is continuous with
respect to $\delta_H$. Then by the monotonicity of $F$ and the
definition of $K+_\infty L$, letting $p\to\infty$, inequality
(\ref{ineq3.1}) (or (\ref{ineq3.3})) yields
\[ F(K+_\infty L)\ge\max\{ F(K), F(L)\}. \]
Furthermore, if $F$ is strictly increasing, equality holds if and
only if either $K$ or $L$ is a subset of the other set.

By the $L_p$ transference principle, we can immediately  obtain  the
$L_p$ Brunn-Minkowski type inequality for quermassintegrals first
established by Firey \cite{Firey}.

\begin{example}
For a convex body $K\in\mathcal{K}^n$, its quermassintegrals
$W_0(K)$, $W_1(K)$, $\ldots$, $W_{n-1}(K)$ are defined by
$W_0(K)=V_n(K)$, and
\[W_{n-j}(K)=\frac{\omega_n}{\omega_{j}}\int_{{\rm G}_{n,j}}V_{j}(K|\xi)d\mu_{j}(\xi), \quad j = 1,\ldots,n-1. \]
Now, the functional
\[{W_{n-j}}^{\frac{1}{j}}:\quad \mathcal{K}^n\to (0,\infty),\quad K\mapsto W_{n-j}(K)^{\frac{1}{j}} \]
is positively homogeneous and increasing. A Brunn-Minkowski
inequality for  $W_{n-j}$ reads as follows: If
 $K,L\in\mathcal{K}^n$ and $0<\alpha<1$, then
\[{W_{n - j}}{((1 - \alpha )K + \alpha L)^{\frac{1}{j}}} \ge (1 - \alpha ){W_{n - j}}{(K)^{\frac{1}{j}}} + \alpha {W_{n - j}}{(L)^{\frac{1}{j}}},\]
with equality if and only if $K$ and $L$ are homothetic. So,
${W_{n-j}}^{{1}/{j}}$ is concave. See, e.g., Gardner \cite[p.
393]{Gardner1}.

Thus, by Corollary \ref{cor3.1}, we directly obtain Firey's $L_p$
Brunn-Minkowski inequality for quermassintegrals: If
$K,L\in\mathcal{K}^n_o$ and $1<p<\infty$, then
\begin{equation}\label{LpBMquermassintegrals}
W_{n-j}(K+_pL)^\frac{p}{j}\ge W_{n-j}(K)^{\frac{p}{j}}+ W_{n-j}(L)^{\frac{p}{j}}.
\end{equation}
From the definition of $W_{n-j}$ and Lemma \ref{lemma2.3}, we know
that ${W_{n-j}}^{{1}/{j}}$ is strictly increasing. Hence, equality
holds in (3.4) if and only if $K$ and $L$ are dilates.
\end{example}

\subsection{Characterizations of equality conditions}

For many $L_p$ Brunn-Minkowski type inequalities, equality only
occurs when the convex bodies are dilates. This phenomenon can be
completely characterized.

\begin{theorem}\label{theorem3.2}
Suppose that  $F:\mathcal{K}^n\to [0,\infty)$ is positively
homogeneous, increasing and concave, and $p\in (1,\infty)$. Then the
following assertions are equivalent.

\noindent (1) For $K, L\in\mathcal{K}^n_o$, the  function
$F_{p;K,L}$ is affine if and only if $K$ and $L$ are dilates.

\noindent (2)  When restricted to $\mathcal{K}^n_o$, the functional
$F$ is strictly increasing.
\end{theorem}

\begin{proof}
The implication ``$(2)\Rightarrow (1)$" is shown by Theorem 3.1.
Next, we prove the implication ``$(1)\Rightarrow (2)$" by
contradiction. Assume that there exist $K_0, L_0\in\mathcal{K}^n_o$
such that $K_0 \subsetneq L_0$ but $F(K_0)=F(L_0)$.

For any $\alpha\in (0,1)$,  let $K_\alpha=(1-\alpha)\cdot_p K_0 +_p
\alpha\cdot_p L_0$. Then
\[ K_0\subset K_\alpha \subset L_0. \]
By the monotonicity of $F$, we have
\[ F(K_0)\le F\left(K_\alpha\right)\le F(L_0). \]
This, together with the assumption $F(K_0)=F(L_0)$, yields
\[ F\left(K_\alpha\right)^p= (1-\alpha)F(K_0)^p+\alpha F(L_0)^p. \]
Thus, assertion (1) implies that $K_0$ and $L_0$ are dilates, say
$K_0=\beta L_0$, for some $\beta>0$. From the positive homogeneity
of $F$ and the assumption $F(K_0)=F(L_0)$ again, we have
\[ F(K_0) = F(\beta L_0) = \beta F(L_0) = F(L_0). \]
Note that $F$ is strictly positive. So, $\beta=1$, and therefore
\[ K_0=L_0, \]
which contradicts the assumption  that $K_0\neq L_0$.
\end{proof}

We say  $F$ is  \emph{translation invariant} if $F(K+x)=F(x)$ for
all $x\in \mathbb{R}^n$.

\begin{theorem}\label{theorem3.3}
Suppose that $F:\mathcal{K}^n\to [0,\infty)$ is translation
invariant, positively homogeneous, increasing and concave, and $p\in
(1,\infty)$. Then the following assertion (1) implies assertion (2).

\noindent (1) For  $K,L\in\mathcal{K}^n$, the function $F_{1;K,L}$
is affine if and only if $K$ and $L$ are homothetic.

\noindent (2)  For $K,L\in\mathcal{K}^n_o$, the function $F_{p;K,L}$
is affine if and only if $K$ and $L$ are dilates.
\end{theorem}

\begin{proof}
Suppose that (1) holds but (2) does not hold, specifically that
there exists an $\alpha_0\in (0,1)$ and $K_0,L_0\in\mathcal{K}^n_o$,
which are not dilates, such that
\[ F\left((1-\alpha_0)\cdot_p K_0 +_p \alpha_0\cdot_p L_0 \right)^p =(1-\alpha_0) F(K_0)^p+\alpha_0 F(L_0)^p. \]
Let
\begin{align*}
&{\alpha _1} = \frac{{{\alpha _0}F{{({L_0})}^p}}}{{(1 - {\alpha _0})F{{({K_0})}^p} + {\alpha _0}F{{({L_0})}^p}}}, \\
&{A_0} = \frac{{{K_0}}}{{F({K_0})}}, \\
&{A_1} = \frac{{{L_0}}}{{F({L_0})}},
\end{align*}
and
\[{A^{(r)}} = (1 - {\alpha _1}){ \cdot _r}{A_0}{ + _r}{\alpha _1}{ \cdot _r}{A_1},\quad {\rm for}\; 1\le r\le p.\]
Clearly, $F(A_0)=F(A_1)=1$.

From Lemma 2.2 and the monotonicity of $F$, we have
\begin{align*}
F\left( {(1 - {\alpha _1}){ \cdot _p}{A_0}{ + _p}{\alpha _1}{ \cdot _p}{A_1}} \right) &\ge F\left( {(1 - {\alpha _1}){A_0} + {\alpha _1}{A_1}} \right) \\
&\ge (1 - {\alpha _1})F({A_0}) + {\alpha _1}F({A_1}).
\end{align*}
Meanwhile, the assumptions yield
\begin{equation}\label{provedkeyineq1}
F\left( {(1 - {\alpha _1}){ \cdot _p}{A_0}{ + _p}{\alpha _1}{ \cdot
_p}{A_1}} \right) = (1-\alpha_1)F(A_0)+\alpha_1F(A_1).
\end{equation}
Hence,
\begin{equation}\label{provedkeyineq2}
F\left( (1-\alpha_1)A_0+\alpha_1 A_1 \right)= (1-\alpha_1) F(A_0) +
\alpha_1 F(A_1).
\end{equation}
Thus, from assertion (1),  there exist $\lambda_0>0$ and
$x_0\in\mathbb{R}^n$ such that
\[ A_0 = \lambda_0 A_1 +x_0. \]

But then, from the positive homogeneity,  translation invariance and
strict positivity of $F$,  we have
\begin{align*}
F(A_1)
&=F(A_0)\\
&= F(\lambda_0 A_1 +x_0)\\
&=\lambda_0 F(A_1).
\end{align*}
Thus,  $\lambda_0=1$.

With $A_0=A_1+x_0$ in hand, for $u\in S^{n-1}$ and $1\le r\le p$, we
have
\[{h_{{A^{(r)}}}}(u) = {\left[ {(1 - {\alpha _1}){{({h_{{A_1}}}(u) + {x_0} \cdot u)}^r} + {\alpha _1}{h_{{A_1}}}{{(u)}^r}} \right]^{\frac{1}{r}}}\]
and
\[{h_{{A_1}}}(u) + {x_0} \cdot u > 0.\]
Thus, for $u\in S^{n-1}$, two observations are in order.

First, if $u\cdot x_0=0$, then for $1\le r_1\le r_2\le p$,
\[{h_{{A^{(1)}}}}(u) = {h_{{A^{({r_1})}}}}(u) = {h_{{A^{({r_2})}}}}(u) = {h_{{A^{(p)}}}}(u).\]

Second, if $u\cdot x_0\ne 0$, then the strict convexity of power
functions  implies that for $1< r_1< r_2< p$,
\[{h_{{A^{(1)}}}}(u) < {h_{{A^{({r_1})}}}}(u) < {h_{{A^{({r_2})}}}}(u) < {h_{{A^{(p)}}}}(u).\]
Consequently, for $1\le r_1<r_2\le p$, we conclude that

\noindent (a) $A^{(1)}\subseteq A^{(r_1)}\subseteq
A^{(r_2)}\subseteq A^{(p)}$ and

\noindent  (b)  $A^{(r_1)}$ and $A^{(r_2)}$ are not homothetic.

Now,  from  (a) it follows that for $\beta\in (0,1)$,
\[ A^{(1)}\subseteq (1-\beta) A^{(1)} +\beta A^{(p)} \subseteq A^{(p)}. \]
Meanwhile, by (\ref{provedkeyineq1}) and (\ref{provedkeyineq2}), we
have
\[ F\left( (1-\alpha_1) A_0+\alpha_1 A_1 \right)= F\left( (1-\alpha_1)\cdot_p A_0 +_p \alpha_1\cdot_p A_1 \right), \]
i.e.,
\[ F\left(A^{(1)}\right)= F\left(A^{(p)}\right). \]
Thus, by the monotonicity of $F$, we obtain
\begin{align*}
F\left( (1-\beta) A^{(1)} +\beta A^{(p)}\right)
&=F\left(A^{(1)}\right)\\
&= F\left(A^{(p)}\right)\\
&=(1-\beta) F\left( A^{(1)} \right)+\beta F\left(A^{(p)} \right).
\end{align*}
Hence, from assertion (1),  $A^{(1)}$ and $A^{(p)}$ are homothetic.
However, this contradicts  (b).
\end{proof}

\begin{example}
The  implication ¡°``$(2)\Rightarrow (1)$" ¡±stated in Theorem
\ref{theorem3.3} does not always hold. This is demonstrated by the
following example, which deals with mean width of convex bodies.

Let $n\ge 2$ and $r<1$,  $r\neq 0$. Define $F: \mathcal{K}^n\to
[0,\infty)$ by
\[F(K) = {\left( {\int_{{S^{n - 1}}} {{w_K}{{(u)}^{ r}}d{\mathcal{H}^{n - 1}}(u)} } \right)^{ \frac{1}{r}}},\]
where $w_K(u) = h_K(u)+h_K(-u)$,  for $ u\in S^{n-1}$,  is the width
function of $K$.

It is obvious  that  $F$ is translation invariant and positively
homogeneous. Minkowski's integral inequality (see \cite[Theorem
198]{BookHardyLittlewoodPolya}) directly yields
\[F\left( {(1 - \alpha )K + \alpha L} \right) \ge (1 - \alpha )F(K) + \alpha F(L), \quad {\rm for}\; \alpha\in (0,1),\]
with equality if and only if  $w_K=\lambda w_L$ for some constant
$\lambda>0$. Note that this may hold without $K$ and $L$ being
homothetic, for example if $K=B^n$ and $L$ is a non-spherical convex
body of the same constant width.

Thus, by the $L_{p}$ transference principle, we obtain
\begin{equation}\label{LpwidthBMineq}
F{((1 - \alpha ){ \cdot _p}K{ + _p}\alpha { \cdot _p}L)^p} \ge (1 - \alpha )F{(K)^p} + \alpha F{(L)^p},
\end{equation}
for any $K,L\in\mathcal{K}^n_o$, $\alpha\in (0,1)$, and $p\in (1,\infty)$.

Finally, we prove that equality holds in (\ref{LpwidthBMineq}) if
and only if $K$ and $L$ are dilates. By Theorem \ref{theorem3.2},
it suffices to prove that $F$ is strictly increasing.

Note that $w_K=h_{K-K}$, for  $K\in\mathcal{K}^n_o$, and
$(K-K)\subset (L-L)$, for any $L\in\mathcal{K}^n_o$ containing $K$.
Hence, if $K \subsetneq L$, then from  continuity of support
functions, there is a nonempty open subset $U\subseteq S^{n-1}$,
such that $w_K(u)< w_L(u)$,  for all $u\in U$. Thus, $F(K)< F(L)$,
for $K,L\in\mathcal{K}^n_o$ with $K\subsetneq L$. That is, the
functional $F$ is strictly increasing on $\mathcal{K}^n_o$.

Hence, the functional $F$ has the required properties.
\end{example}

\vskip 25pt
\section{\bf Applications of the $\boldsymbol{L_p}$ transference principle}
\vskip 10pt

In this section, we aim to demonstrate the effectiveness and
practicability of the $L_p$ transference principle. As
illustrations, several new   $L_{p}$ Brunn-Minkowski  type
inequalities are established.

\subsection{An application to mixed volumes}

The mixed volume $V: \mathcal{K}^n \to [0,\infty)$ is a nonnegative
and symmetric functional  such that
\[ V_n\left( \lambda_1 K_1 +\cdots +\lambda_m K_m \right)=\sum_{i_1,\cdots,i_n=1}^m V(K_{i_1},\ldots,K_{i_m})\lambda_{i_1}\cdots\lambda_{i_m}, \] for
$K_1,\ldots,K_m\in\mathcal{K}^n$ and $\lambda_1,\ldots,\lambda_m>0$.
See, e.g., Schneider \cite[Chapter 5]{BookSchneider}.

Write $V(K,j;K_{j+1},\ldots,K)$ for mixed volume
$V(K,\cdots,K,K_{j+1},\ldots,K_n)$ with $j$ copies of $K$. If
$j=n$, then $V(K,j;K_{j+1},\ldots,K_n)$ is $V_n(K)$. The
classical Brunn-Minkowski inequality has a natural extension to
mixed volumes as follows (see, e.g., Schneider \cite[Theorems 7.4.5,
7.4.6 and 7.6.9]{BookSchneider}).

\begin{proposition}\label{lemma4.1}
Suppose  $K, L, K_j,\ldots, K_n\in\mathcal{K}^n$, and $j\in
\{2,\ldots, n\}$. Then
\[V{(K + L,j;{K_{j + 1}}, \ldots ,{K_n})^{\frac{1}{j}}} \ge V{(K,j;{K_{j + 1}}, \ldots ,{K_n})^{\frac{1}{j}}} + V{(L,j;{K_{j + 1}}, \ldots ,{K_n})^{\frac{1}{j}}}.\]
If $j=n$, or if $2 \le j\le n-1$ and $K_{j+1},\ldots, K_n$ are
smooth, then equality holds if and only if $K$ and $L$ are homothetic.
\end{proposition}

From Proposition \ref{lemma4.1} and the $L_p$ transference
principle, we  obtain the following result.

\begin{theorem}\label{theorem4.1}
Suppose $K, L\in\mathcal{K}^n_o$,  $K_j,\ldots, K_n\in\mathcal{K}^n$,
 $p\in (1,\infty)$, and $j\in \{2,\ldots, n\}$. Then
\[V{(K{ + _p}L,j;{K_{j + 1}}, \ldots ,{K_n})^{\frac{p}{j}}} \ge V{(K,j;{K_{j + 1}}, \ldots ,{K_n})^{\frac{p}{j}}} + V{(L,j;{K_{j + 1}}, \ldots ,{K_n})^{\frac{p}{j}}}.\]
If $j=n$,  or if $2 \le j\le n-1$ and $K_{j+1},\ldots, K_n$ are
smooth, then equality holds if and only if $K$ and $L$ are dilates.
\end{theorem}

If $j=n$, then the previous inequality reduces to (\ref{LpBMineq}).
If
$1\le j\le n-1$ and $K_{j+1}=\cdots=K_n=B^n$, then
the previous inequality becomes (\ref{LpBMquermassintegrals}).

\begin{proof}
We only need consider the case $1\le j\le n-1$. For $K\in\mathcal{K}^n$,
define
\[F(K) = V{(K,j;{K_{j + 1}}, \ldots ,{K_n})^{\frac{1}{j}}}.\]
Then $F$ is positively homogeneous and increasing (see, e.g.,
Schneider \cite[(5.25), p. 282]{BookSchneider}). Since convex bodies
are $n$-dimensional, from Theorem 5.1.8 of Schneider
\cite[p. 283]{BookSchneider}, $F$ is strictly positive. Meanwhile,
Proposition \ref{lemma4.1} implies that $F$ is concave.

Hence,  from the $L_{p}$ transference principle, it follows that
\begin{equation}\label{generalLpBMineq}
F\left((1-\alpha)\cdot_p K +_p \alpha \cdot_p L \right)^p \ge
(1-\alpha) F(K)^p +\alpha F(L)^p,
\end{equation}
for $K,L\in\mathcal{K}^n_o$ and $0<\alpha<1$.

Assume $2\le j \le n-1$, and the bodies $K_{j+1},\ldots, K_n$ are smooth.
Note that $F$ is translation invariant. Thus, by Proposition
\ref{lemma4.1} and Theorem \ref{theorem3.3}, equality holds in
(\ref{generalLpBMineq}) if and only if $K$ and $L$ are dilates.
\end{proof}

\subsection{An application to moments of inertia}

From classic mechanics, we know that for each convex body $K$ in
$\mathbb{R}^n$, its moment of inertia, $I(K)$,  is defined by
\[ I(K)=\int_{K}|x-c_K|^2dx, \]
where $c_K$ denotes the centroid of $K$.

Proposition \ref{lemma4.2} was originally established by Hadwiger
\cite{BMineqHadwiger}.
\begin{proposition}\label{lemma4.2}
Suppose $K,L\in\mathcal{K}^n$. Then
\[ I(K+L)^{\frac{1}{n+2}} \ge I(K)^{\frac{1}{n+2}} + I(L)^{\frac{1}{n+2}}. \]
\end{proposition}

From Proposition \ref{lemma4.2} and the $L_p$ transference
principle, we obtain the following result.
\begin{theorem}\label{theorem4.2}
Suppose that  $K, L\in\mathcal{K}^n_o$  are origin-symmetric and
$1<p<\infty$. Then
\begin{equation}\label{LpBMineqDualVolume}
I(K+_pL)^{\frac{p}{n+2}} \ge I(K)^{\frac{p}{n+2}} +
I(L)^{\frac{p}{n+2}},
\end{equation}
with equality if and only if $K$ and $L$ are dilates.
\end{theorem}

\begin{proof}
For $K\in\mathcal{K}^n$, define
\[ F(K)=I(K)^{\frac{1}{n+2}}. \]
Obviously, $F$ is positively homogeneous. From Proposition
\ref{lemma4.2},
 $F$ is concave. Moreover, if $K$ is
origin-symmetric, then the centroid $c_K$ of $K$ is at the origin,
and then
\[ F(K)=\left( \int_K |x|^2dx \right)^{\frac{1}{n+2}}. \]

When the domain of $F$ is restricted to the subset
$\mathcal{K}^n_{o,s}\subseteq \mathcal{K}^n_o$, the class of
origin-symmetric convex bodies,  then $F:\mathcal{K}^n_{o,s}\to
(0,\infty)$ is strictly increasing.

Hence,  from the $L_{p}$ transference principle, we obtain the theorem.
\end{proof}

For an origin-symmetric convex body $K$,  its isotropic constant
$L_K$ is defined by
\begin{equation}\label{defisotropicconst}
{L_K}^2 = \frac{1}{n}\min \left\{
{\frac{{I(TK)}}{{{V_n}{{(K)}^{\frac{{n + 2}}{n}}}}}:T \in
{\rm{SL}}(n)} \right\}.
\end{equation}
For more information on isotropic constants, we refer to Milman and
Pajor \cite{Milman}.

From (\ref{LpBMineqDualVolume}) and (\ref{defisotropicconst}), we
obtain the following result.

\begin{corollary}
Suppose that  $K_0,K_1\in\mathcal{K}^n_o$ are origin-symmetric, and
$1\le p<\infty$. Then
\[V_n{\left( {{K_0}{ + _p}{K_1}} \right)^{\frac{p}{{n}}}}{L_{{K_0}{ + _p}{K_1}}}^{\frac{2p}{{n + 2}}} \ge V_n{\left( {{K_0}} \right)^{\frac{p}{{n}}}}{L_{{K_0}}}^{\frac{2p}{{n + 2}}} + V_n{\left( {{K_1}} \right)^{\frac{p}{{n}}}}{L_{{K_1}}}^{\frac{2p}{{n + 2}}}.\]
\end{corollary}

\subsection{An application to affine quermassintegrals}

For a convex body $K\in\mathcal{K}^n$, Hadwiger \cite[p.
267]{BookHadwiger} introduced the harmonic quermassintegrals
$\hat{W}_0(K)$, $\hat{W}_1(K)$, $\ldots$, $\hat{W}_{n-1}(K)$,
defined by $\hat{W}_0(K)=V_n(K)$, and
\[ \hat{W}_j(K)=\frac{\omega_n}{\omega_{n-j}}\left(\int_{{\rm
G}_{n,n-j}}V_{n-j}(K|\xi)^{-1}d\mu_{n-j}(\xi)\right)^{-1}, \quad
j=1, \ldots, n-1.\]
See also Gardner \cite[p. 382]{BookGardner},
Schneider \cite[p. 514]{BookSchneider}, and Lutwak
\cite{BMLutwak1, BMLutwak3}.  Nearly thirty years later,
Lutwak \cite{BMLutwak1, BMLutwak3} introduced the affine
quermassintegrals $\Phi_0(K)$, $\Phi_1(K)$, $\ldots$,
$\Phi_{n-1}(K)$,  defined by $\Phi_0(K)=V_n(K)$, and
\[ \Phi_j(K)=\frac{\omega_n}{\omega_{n-j}}\left(\int_{{\rm G}_{n,n-j}}V_{n-j}(K|\xi)^{-n}d\mu_{n-j}(\xi)\right)^{-\frac{1}{n}}, \quad j=1,\ldots,n-1. \]
Note that all the  $\Phi_j(K)$ are affine invariant, i.e.,
$\Phi_j(TK)=\Phi_j(K)$,  for all $T \in {\rm{SL}}(n)$. See Grinberg
\cite{Grinberg}. For more information, we refer to Gardner
\cite{GadnerAffineQuermass} and Dafnis and Paouris
\cite{AffinequermassDafnisPaouris}.

Hadwiger \cite[p. 268]{BookHadwiger} and Lutwak \cite{BMLutwak1}
established the following Brunn-Minkowski type inequalities for
harmonic quermassintegrals and affine quermassintegrals,
respectively.

\begin{proposition}\label{lemma4.3}
Suppose  $K, L\in\mathcal{K}^n$ and $j\in \{1,\ldots,n-1\}$. Then
\[ \hat{W}_j(K+L)^{\frac{1}{n-j}}\ge \hat{W}_j(K)^{\frac{1}{n-j}}+ \hat{W}_j(L)^{\frac{1}{n-j}} \]
and
\[ \Phi_j(K+L)^{\frac{1}{n-j}}\ge \Phi_j(K)^{\frac{1}{n-j}}+ \Phi_j(L)^{\frac{1}{n-j}}.\]
If $j=n-1$,  equality holds in each inequality if and only if
$w_K=\lambda w_L$ for some constant $\lambda>0$. If $1\le j<n-1$,
equality holds in each inequality if and only if $K$ and $L$ are
homothetic.
\end{proposition}

From Proposition \ref{lemma4.3} and the $L_p$ transference
principle, we  obtain the following result.

\begin{theorem}\label{theorem4.3}
Suppose $K,L\in\mathcal{K}^n_o$ and $1<p<\infty$. Then
\[ \hat{W}_j(K+_p L)^{\frac{p}{n-j}}\ge \hat{W}_j(K)^{\frac{p}{n-j}}+ \hat{W}_j(L)^{\frac{p}{n-j}}  \]
and
\[ \Phi_j(K+_p L)^{\frac{p}{n-j}}\ge \Phi_j(K)^{\frac{p}{n-j}}+ \Phi_j(L)^{\frac{p}{n-j}}. \]
Equality holds in each inequality if and only if $K$ and $L$ are
dilates.
\end{theorem}

\begin{proof}
We  prove this theorem for affine quermassintegrals. The proof for
harmonic quermassintegrals is similar. For $K\in\mathcal{K}^n$,
define
\[ F(K)=\Phi_j(K)^{\frac{1}{n-j}}.\]
Then  $F$ is positively homogeneous. From the definition of $\Phi_j$
and Lemma \ref{lemma2.3},  $F$ is strictly increasing. From
Proposition \ref{lemma4.3}, $F$ is concave. Hence, from the $L_{p}$
transference principle and Theorem \ref{theorem3.2}, we obtain the
theorem .
\end{proof}

\subsection{An application to projection bodies}

For a convex body $K\in\mathcal{K}^n$, its mixed projection bodies
$\Pi_{0} K$, $\Pi_1 K$, $\ldots$, $\Pi_{n-1}K$ are defined by
\[h_{\Pi_i K}(u)=W^{(n-1)}_i(K|u^\bot),\]
for $u\in S^{n-1}$ and $i\in\{0,1, \ldots, n-1\}$, where
$W^{(n-1)}_i(K|u^\bot)$ denotes the $i$th quermassintegral of
$K|u^\bot$ defined in the subspace $u^\bot$. For more information about
mixed projection bodies, we refer to Gardner \cite[p.
185]{BookGardner},  Lutwak \cite{BMLutwak2, BMLutwak4}, Parapatits
and Schuster \cite{Projectionbody}, and Schneider \cite[p.
578]{BookSchneider}.

In \cite{BMLutwak4}, Lutwak established the following
Brunn-Minkowski type inequality for projection bodies.

\begin{proposition}\label{lemma4.4}
Suppose $K,L\in\mathcal{K}^n$, $j\in\{1,\ldots,n\}$, and
$k\in\{1,\ldots,n-1\}$. Then
\[W_{n - j}{\left( {{\Pi _{n - 1 - k}}(K + L)} \right)^{\frac{1}{{jk}}}} \ge W_{n - j}{\left( {{\Pi _{n - 1 - k}}K} \right)^{\frac{1}{{jk}}}} + W_{n - j}{\left( {{\Pi _{n - 1 - k}}L} \right)^{\frac{1}{{jk}}}},\]
with equality if and only if $K$ and $L$ are homothetic.
\end{proposition}

From Proposition \ref{lemma4.4} and the $L_p$ transference
principle, we  obtain the following result.

\begin{theorem}\label{theorem4.4}
Suppose $K,L\in\mathcal{K}^n_o$, $j\in\{1,\ldots,n\}$,
$k\in\{1,\ldots,n-1\}$, and $p\in (1,\infty)$. Then
\begin{equation}\label{LpprojectionBMineq1}
{W_{n - j}}{\left( {{\Pi _{n - 1 - k}}(K +_p L)}
\right)^{\frac{p}{{jk}}}} \ge {W_{n - j}}{\left( {{\Pi _{n - 1 -
k}}K} \right)^{\frac{p}{{jk}}}} + {W_{n - j}}{\left( {{\Pi _{n - 1 -
k}}L} \right)^{\frac{p}{{jk}}}},
\end{equation}
with equality if and only if $K$ and $L$ are dilates.
\end{theorem}

\begin{proof}
Some facts about mixed projection bodies are in order.

First, for each $K\in\mathcal{K}^n$, the compact convex set
$\Pi_{n-1-k} K$ is a convex body. Indeed,  for all $u\in S^{n-1}$,
since $K|u^\bot$ is $(n-1)$-dimensional, it follows that
$W^{(n-1)}_{n-1-k}(K|u^\bot)>0$, i.e., $h_{\Pi_{n-1-k}K}(u)>0$.

Second, $\Pi_{n-1-k}(\lambda K)=\lambda^k\Pi_{n-1-k} K$, for all
$\lambda>0$. Indeed, for all $u\in S^{n-1}$,
\begin{align*}
{h_{{\Pi _{n - 1 - k}}(\lambda K)}}(u) &= W_{n - 1 - k}^{(n - 1)}\left( {(\lambda K)|{u^ \bot }} \right) \\
&= {\lambda ^k}W_{n - 1 - k}^{(n - 1)}\left( {K|{u^ \bot }} \right) \\
&= {\lambda ^k}{h_{{\Pi _{n - 1 - k}}(\lambda K)}}(u).
\end{align*}

Third, if $K,L\in\mathcal{K}^n$ and $K\subseteq L$, then
$\Pi_{n-1-k} K\subseteq \Pi_{n-1-k} L$. Indeed, for all $u\in
S^{n-1}$, since $K|u^\bot\subseteq L|u^\bot$,  it follows that
$W_{n-1-k}^{(n-1)}(K|u^\bot) \le W_{n-1-k}^{(n-1)}(L|u^\bot)$, i.e.,
$h_{\Pi_{n-1-k} K}(u)\le h_{\Pi_{n-1-k} L}(u)$.

Fourth, ${\Pi _{n - 1 - k}}(K + x) = {\Pi _{n - 1 - k}}K$, for all
$x\in\mathbb{R}^n$. Indeed, for all $u\in S^{n-1}$,
\begin{align*}
{h_{{\Pi _{n - 1 - k}}(K + x)}}(u) &= W_{n - 1 - k}^{(n - 1)}\left( {(K + x)|{u^ \bot }} \right) \\
&= W_{n - 1 - k}^{(n - 1)}\left( {K|{u^ \bot } + x|{u^ \bot }} \right) \\
&= W_{n - 1 - k}^{(n - 1)}\left( {K|{u^ \bot }} \right) \\
&= {h_{{\Pi _{n - 1 - k}}K}}(u).
\end{align*}

Hence, the functional $F =W_{n-j}(\Pi_{n-1-k}(\cdot))^{{1}/{jk}}$
over $\mathcal{K}^n$ is strictly positive, positively homogeneous,
increasing, and translation invariant. Meanwhile, Proposition
\ref{lemma4.4} implies that $F$ is concave.

Hence,  from the $L_{p}$ transference principle and Proposition
\ref{lemma4.4} together with Theorem \ref{theorem3.3}, we obtain the
theorem.
\end{proof}

\subsection{An application to capacities}

The $q$-capacity of a convex body $K$ in $\mathbb{R}^n$, for $1\le
q<n$, is
\[{\rm Cap}_q(K) = \inf \left\{ {\int_{{\mathbb{R}^n}} {|\nabla f|^qdx} } \right\},\]
where the infimum is taken over all nonnegative functions $f$ such
that $f\in L^{\frac{nq}{n-q}}(\mathbb{R}^n)$, $\nabla f\in
L^q(\mathbb{R}^n;\mathbb{R}^n)$, and $K$ is contained in the
interior of $\{x: f(x)\ge 1\}$.

The following is the remarkable capacitary Brunn-Minkowski
inequality.

\begin{proposition}\label{lemma4.5}
Suppose $K,L\in\mathcal{K}^n$, and $1\le q<n$. Then
\begin{equation}\label{capacitaryBMineq}
{\rm Cap}_q{(K + L)^{\frac{1}{{n - q}}}} \ge {\rm
Cap}_q{(K)^{\frac{1}{{n - q}}}} + {\rm Cap}_q{(L)^{\frac{1}{{n - q}}}},
\end{equation}
with equality if and only if $K$ and $L$ are homothetic.
\end{proposition}

Borell \cite{BMineqBorell} first established
(\ref{capacitaryBMineq}) for the case $q=2$ (the Newtonian
capacity), and the equality condition was proved by Caffarelli,
Jerison, and Lieb \cite{BMineqCaffarelliJerison}. When $1<q<n$, the
 inequality  was proved by Colesanti and Salani
\cite{BMineqColesanti1}. The case $q=1$ is just the Brunn-Minkowski
inequality for surface area of convex bodies:
\[W_1(K+L)^\frac{1}{n-1} \ge W_1(K)^\frac{1}{n-1} + W_1(L)^\frac{1}{n-1}, \quad K,L\in\mathcal{K}^n,\]
due to the fact ${\rm Cap}_1(K)=\mathcal{H}^{n-1}(\partial K)=n
W_1(K)$, for $K\in\mathcal{K}^n$.

For more information on the role of capacity in the Brunn-Minkowski
theory and its dual, we refer to Gardner and Hartenstine
\cite{Gardner3} and the references within.

From Proposition \ref{lemma4.5} and  the $L_p$ transference principle, we
obtain the following result.
\begin{theorem}\label{theorem4.5}
Suppose $K,L\in\mathcal{K}^n_o$, $1\le q<n$, and $1<p<\infty$. Then
\[{\rm Cap}_q{(K{ + _p}L)^{\frac{p}{{n - q}}}} \ge {\rm Cap}_q{(K)^{\frac{p}{{n - q}}}} + {\rm Cap}_q{(L)^{\frac{p}{{n - q}}}}, \]
with equality if and only if $K$ and $L$ are dilates.
\end{theorem}

\begin{proof}
From Evans and Gariepy \cite[pp. 150-151]{BookEvansGariepy}, the
functional
\[F= {\rm Cap}_q (\cdot)^{\frac{1}{n-q}}: \mathcal{K}^n\to (0,\infty)\]
is positively homogeneous, increasing, and translation
invariant. Meanwhile, Proposition \ref{lemma4.5} implies that $F$ is
concave.

Hence,  from the $L_{p}$ transference principle and Proposition
\ref{lemma4.5} together with Theorem \ref{theorem3.3}, Theorem
\ref{theorem4.5} is obtained.
\end{proof}

\begin{center}
\textbf{Acknowledgement}
\end{center}

\vskip 10pt

We thank the referee very much for the extremely thorough and
helpful report on our manuscript.

\bibliographystyle{amsplain}

\end{document}